\documentclass[12pt,a4paper]{article}

\usepackage[dvips]{color}
\usepackage{color}
\usepackage{xcolor}

\usepackage{amsfonts}
\usepackage{amsfonts}
\usepackage{amsfonts}
\usepackage{amsfonts}
\usepackage{amsfonts}
\usepackage{mathrsfs}
\usepackage{amsfonts}
\usepackage{mathrsfs}
\usepackage{amsfonts}
\usepackage{amsfonts}
\usepackage{mathrsfs}
\usepackage{mathrsfs}
\usepackage{amsfonts}
\usepackage{amsfonts}
\usepackage{amsfonts}
\usepackage{amsfonts}
\usepackage{mathrsfs}

\usepackage{amsfonts}
\usepackage{mathrsfs}
\usepackage{mathrsfs}
\usepackage{mathrsfs}
\usepackage{mathrsfs}
\usepackage{amsfonts}
\usepackage{mathrsfs}
\usepackage{mathrsfs}
\usepackage{amsfonts}
\usepackage{amsfonts}
\usepackage{amsfonts}
\usepackage{amsfonts}
\usepackage{amsfonts}
\usepackage{amsfonts}
\usepackage{amsfonts}
\usepackage{amsfonts}
\usepackage{amsfonts}
\usepackage{mathrsfs}
\usepackage{amsfonts}
\usepackage{mathrsfs}
\usepackage{amsfonts}
\usepackage{mathrsfs}
\usepackage{amsfonts}
\usepackage{amsfonts}
\usepackage{amsfonts}
\usepackage{amsfonts}
\usepackage{amsfonts}
\usepackage{amsfonts}
\usepackage{amsfonts}
\usepackage{amsfonts}
\usepackage{amsfonts}
\usepackage{amsfonts}
\usepackage{amsfonts}
\usepackage{amsfonts}
\usepackage{mathrsfs}
\usepackage{mathrsfs}
\usepackage{mathrsfs}
\usepackage{mathrsfs}
\usepackage{mathrsfs}
\usepackage{mathrsfs}
\usepackage{mathrsfs}
\usepackage{mathrsfs}
\usepackage{mathrsfs}
\usepackage{amsfonts}
\usepackage{amsfonts}
\usepackage{amsfonts}
\usepackage{amsfonts}
\usepackage{amsfonts}
\usepackage{amsfonts}
\usepackage{mathrsfs}
\usepackage{amsfonts}
\usepackage{amsfonts}
\usepackage{amsfonts}
\usepackage{amsfonts}
\usepackage{amsfonts}
\usepackage{amsfonts}
\usepackage{amsfonts}
\usepackage{amsfonts}
\usepackage{amsfonts}
\usepackage{amsfonts}
\usepackage{amsfonts}
\usepackage{amsfonts}
\usepackage{mathrsfs}
\usepackage{amsfonts}
\usepackage{mathrsfs}
\usepackage{amsfonts}
\usepackage{amsfonts}
\usepackage{amsfonts}
\usepackage{mathrsfs}
\usepackage{amsfonts}
\usepackage{amsfonts}
\usepackage{amsfonts}
\usepackage{amsfonts}
\usepackage{amsmath}
\usepackage{mathrsfs}
\usepackage{amsfonts}
\usepackage{amsfonts}
\usepackage{amsfonts}
\usepackage{amsfonts}
\usepackage{amsfonts}
\usepackage{amsfonts}
\usepackage{amsfonts}
\usepackage{amsfonts}
\usepackage{amsfonts}
\usepackage{amsfonts}
\usepackage{amsfonts}
\usepackage{mathrsfs}
\usepackage{amsfonts}
\usepackage{amsfonts}
\usepackage{amsfonts}
\usepackage{amsfonts}
\usepackage{amsmath}
\usepackage{mathrsfs}
\usepackage{mathrsfs}
\usepackage{mathrsfs}
\usepackage{mathrsfs}
\usepackage{mathrsfs}
\usepackage{mathrsfs}
\usepackage{mathrsfs}
\usepackage{mathrsfs}
\usepackage{amsfonts}
\usepackage{amsfonts}
\usepackage{amsfonts}
\usepackage{amsfonts}
\usepackage{amsfonts}
\usepackage{amsfonts}
\usepackage{amsfonts}
\usepackage{amsfonts}
\usepackage{amsfonts}
\usepackage{amsfonts}
\usepackage{amsfonts}
\usepackage{amsfonts}
\usepackage{mathrsfs}
\usepackage{amsfonts}
\usepackage{mathrsfs}
\usepackage{amsfonts}
\usepackage{amsfonts}
\usepackage{amsfonts}
\usepackage{amsfonts}
\usepackage{amsfonts}
\usepackage{amsfonts}
\usepackage{amsfonts}
\usepackage{amsfonts}
\usepackage{amsfonts}
\usepackage{amsfonts}
\usepackage{amsfonts}
\usepackage{amsfonts}
\usepackage{amsfonts}
\usepackage{amsfonts}
\usepackage{amsfonts}
\usepackage{amsfonts}
\usepackage{amsfonts}
\usepackage{amsfonts}
\usepackage{mathrsfs}
\usepackage{mathrsfs}
\usepackage{mathrsfs}
\usepackage{mathrsfs}
\usepackage{amsfonts}
\usepackage{amsfonts}
\usepackage{amsfonts}
\usepackage{amsfonts}
\usepackage{amsfonts}
\usepackage{amsfonts}
\usepackage{amsfonts}
\usepackage{amsfonts}
\usepackage{amsfonts}
\usepackage{amsfonts}
\usepackage{amsfonts}
\usepackage{amsfonts}
\usepackage{amsfonts}
\usepackage{amsfonts}
\usepackage{amsfonts}
\usepackage{amsfonts}
\usepackage{amsfonts}
\usepackage{mathrsfs}
\usepackage{mathrsfs}
\usepackage{amsfonts}
\usepackage{amsfonts}
\usepackage{amsfonts}
\usepackage{amsfonts}
\usepackage{amsfonts}
\usepackage{amsfonts}
\usepackage{amsfonts}
\usepackage{amsfonts}
\usepackage{amsfonts}
\usepackage{amsfonts}
\usepackage{amsfonts}
\usepackage{amsfonts}
\usepackage{amsfonts}
\usepackage{mathrsfs}
\usepackage{mathrsfs}
\usepackage{amsfonts}
\usepackage{amsfonts}
\usepackage{amsfonts}
\usepackage{amsfonts}
\usepackage{amsfonts}
\usepackage{amsfonts}
\usepackage{mathrsfs}
\usepackage{mathrsfs}
\usepackage{amsfonts}
\usepackage{amsfonts}
\usepackage{amsfonts}
\usepackage{amsfonts}
\usepackage{amsfonts}
\usepackage{amsfonts}
\usepackage{amsfonts}
\usepackage{amsfonts}
\usepackage{amsfonts}
\usepackage{amsfonts}
\usepackage{amsfonts}
\usepackage{amsfonts}
\usepackage{amsfonts}
\usepackage{amsfonts}
\usepackage{mathrsfs}
\usepackage{amsfonts}
\usepackage{amsfonts}
\usepackage{amsfonts}
\usepackage{amsfonts}
\usepackage{amsfonts}
\usepackage{amsfonts}
\usepackage{amsfonts}
\usepackage{amsfonts}
\usepackage{amsfonts}
\usepackage{amsfonts}
\usepackage{amsfonts}
\usepackage{amsfonts}
\usepackage{amsfonts}
\usepackage{amsfonts}
\usepackage{amsfonts}
\usepackage{amsfonts}
\usepackage{amsmath}
\usepackage{mathrsfs}
\usepackage{amsfonts}
\usepackage{amsfonts}
\usepackage{amsfonts}
\usepackage{amsfonts}
\usepackage{amsmath}
\usepackage{mathrsfs}
\usepackage{amsfonts}
\usepackage{mathrsfs}
\usepackage{mathrsfs}
\usepackage{mathrsfs}
\usepackage{mathrsfs}
\usepackage{amsmath}
\usepackage{mathrsfs}
\usepackage{amsfonts}
\usepackage{color}
\usepackage{mathrsfs}

\usepackage{stmaryrd}
\usepackage{amsmath}
\usepackage{amssymb}

\usepackage{bbm}
\usepackage{mathrsfs}
\usepackage{graphicx}

\usepackage{enumerate}
\usepackage{thm}

\makeatletter
\def\tank#1{\protected@xdef\@thanks{\@thanks
        \protect\footnotetext[0]{#1}}}
\def\bigfoot{

    \@footnotetext}
\makeatother

\newcommand{\ea}{\end{array}}
\newtheorem{thm}{Theorem}[section]
\newtheorem{prop}{Proposition}[section]

\newtheorem{ass}{Assumption}[section]

\numberwithin{equation}{section}

{\theorembodyfont{\rmfamily}

\newenvironment{proof}{Proof}{\hfill $\Box$}

\def\RR{\mathbb{R}}
\def\PP{\mathbb{P}}

\def\NN{\mathbb{N}}

\def\cB{{\mathcal B}}

\def\si{{\sigma}}

\def\si{{\sigma}}

\def\si{{\sigma}}

\setcounter{equation}{0}

\allowdisplaybreaks

\begin{document}
\title{\Large \bf Weak irreducibility of stochastic delay differential equation driven by pure jump noise}
\date{}

\author{
{Hao Yang}$^1$\footnote{E-mail:yanghao@hfut.edu.cn}~~~{Jian Wang}$^2$\footnote{E-mail:20230078@hznu.edu.cn}
\\
 \small 1.School of Mathematics, Hefei University of Technology, Hefei, Anhui 230009, China.\\
 \small 2. School of Mathematics, Hangzhou Normal University, Hangzhou 311121, China. 
}

\maketitle

\begin{center}
\begin{minipage}{130mm}
{\bf Abstract:}
In this paper, we study the weak irreducibility of stochastic delay differential equations(SDDEs) driven by pure jump noise. The main contribution of this paper is to provide a concise proof of weak irreducibility, releasing condition (A1-3) in Assumption 2.1 from the literature \cite{WYZZ1}. As an application, we derive the weak irreducibility of SDDEs with weakly dissipative coefficients. An important novelty of this paper is to allow the driving noises to be degenerate. This closes the gap of the irreducibility of SDDEs driven by pure jump noise.


\vspace{3mm} {\bf Keywords:}
Weak irreducibility; pure jump noise; weakly dissipative; degenerate

\vspace{3mm} {\bf AMS Subject Classification (2020):}
34K50; 60H10; 60G51; 37A50.
\end{minipage}
\end{center}

\newpage

\renewcommand\baselinestretch{1.2}
\setlength{\baselineskip}{0.28in}
\section{Introduction and motivation}\label{Intr}
Let $H$ be a topological space with Borel $\sigma$-field $\mathcal{B}(H)$, and let $\mathbb{X}:=\{X^{\xi}(t),t\geq0; \xi\in H\}$ be an $H$-valued  Markov process on some
probability space $(\Omega,\mathcal{F},\PP)$.
$\mathbb{X}$ is said to be  weakly irreducible (also called accessible) to
$\xi_0\in H$ if the resolvent $R_{\lambda}, \lambda>0$ satisfies
\begin{equation*}
R_{\lambda}(\xi,U):=\lambda \int_0^{\infty}e^{-\lambda t}\PP(X^\xi(t)\in U)dt>0
\end{equation*}
for all $\xi\in H$ and all neighborhoods $U$ of $\xi_0$, where $\lambda > 0$ is arbitrary.

%

The significance of studying weak irreducibility lies in its contribution to the ergodicity of Markov processes on Polish spaces. Indeed, weak irreducibility, when combined with a smoothing property (such as the asymptotic strong Feller property \cite{Hairer M 2006}) or contractivity (e-property \cite{Kapica} or eventual continuity \cite{Gong-2025}) of the Markov semigroup, implies that there can be at most one invariant probability measure. It is noteworthy that eventual continuity is a weaker condition than both the asymptotic strong Feller property and the e-property. Specifically, there exist Markov-Feller processes that exhibit eventual continuity but do not satisfy either the asymptotic strong Feller property or the e-property; further details can be found in references \cite{Gong-2024,Gong-2025,Liu-2024}.

Now we survey the existing results concerning the weak  irreduciblility of stochastic
systems driven by pure jump noise. For the case of stochastic differential equations(SDEs), the authors in \cite{kulyk} derived the support property of SDEs driven by pure jump noise under local Lipschitz condition, which strongly implies weak  irreduciblility. For such support property, one can also refer to \cite{ishikawa, Kunita, simon, simon1}. For the case of SPDEs,  the authors in
\cite{Xulihu} obtained the weak  irreduciblility to zero of stochastic real valued Ginzburg-Landau
equation on torus $\mathbb{T}=\mathbb{R}\setminus \mathbb{Z}$ driven by cylindrical symmetric $\alpha$-stable process with
$\alpha\in (1,2)$. In \cite{yang}, the authors extended the result in \cite{Xulihu} to $n$ dimensions. The authors in \cite{deng} obtained
the weak  irreduciblility to zero of a class of semilinear SPDEs with Lipschitz coeﬃcients
driven by multiplicative pure jump noise. Since their methods of obtaining the weak  irreduciblility are basically along the
same lines as that for the Gaussian case, the driving noises they considered are
a class of subordinated Wiener processes. In our recent papers \cite{WYZZ, WYZZ1}, we developed new criterions for the irreducibility of
SPDEs driven by multiplicative pure jump noise and weak  irreduciblility of SPDEs driven by additive pure jump noise. The driven noises in \cite{WYZZ, WYZZ1} could be very weak, including a large class of Lévy processes with heavy tails and a large class of compound
Poisson processe, etc.

Delay differential equations arise from evolutionary phenomena with time delay in fields such as physics, biology, and engineering \cite{Hale}. Specifically, such equations have been used to describe the dynamic behaviors of materials with thermal memory, biochemical reactions, and population models. For instance, in population dynamics systems, delayed Lotka-Volterra models, chemotaxis models, SIR models are often employed to depict the evolution trends of animal and microbial systems over time. To consider the irreducibility of such SDDEs, it is generally necessary to consider the infinite-dimensional fragment process, so we need to analyze the property of whole path. In this case, it is a difficult problem to study the irreducibility of solution maps driven by pure jump L\'evy noise. To the best of our knowledge, there are no results on the irreducibility of SDDEs driven by pure jump noise. This strongly motivates the current paper.

In this paper, we establish the weak  irreduciblility of SDDEs driven by additive pure jump noise; see Theorem \ref{thm1} in Section 2. The main contribution of this paper is to provide a concise proof of weak irreducibility, releasing condition (A1-3) in Assumption 2.1 from the literature \cite{WYZZ1}. As an application, we obtain the weak  irreduciblility for SDDEs driven by pure jump noise with weakly dissipative coefficients; see Proposition \ref{prop 1}. We remark that an important novelty is to allow the driving noises to be degenerate and we only require the L\'evy measure to be symmetric. This closes the gap of the irreducibility of SDDEs
driven by pure jump noise.

The organization of the paper is as follows.
In Section 2, we introduce the precise
framework and the main result. In Section 3, we provide some applications to
SDDEs, including the weak  irreduciblility of a class of SDDEs with weakly dissipative coefficients.

\section{The framework and main results}
Unless otherwise specified, we will use the following notation. Let
$|\cdot|$ denote  the Euclidean norm on  $\mathbb{R}^d$ with the  corresponding  inner product  $\langle\cdot,\cdot\rangle$ and $(\Omega,{\mathcal{F}},P)$ be a complete probability space with a filtration $\{\mathcal{F}_t\}$ satisfying the usual conditions. For a vector or matrix $A$, its transpose is denoted by $A^T$ and $|A|={\rm Trace}(A^TA)$ denotes its trace norm. Denote by $D((-\infty,0];\mathbb{R}^n)$ the family of function from $(-\infty,0]$ to $\mathbb{R}^n$. In this paper, we choose the phase space $D_r~(r>0)$ defined by
\begin{equation*}
  D_r=\{\varphi\in D((-\infty,0];\mathbb{R}^n):\lim_{\theta\rightarrow -\infty}e^{r\theta}\varphi(\theta)~\rm{exists~in}~\mathbb{R}^n\},
\end{equation*}
which is a Banach space with norm $\|\varphi\|_r=\sup_{-\infty<\theta\leq 0}e^{r\theta}|\varphi(\theta)|<\infty.$ Note that all bounded and continuous functions are contained in this space, and for any $0<r_1<r_2<\infty$, $D_{r_1}\subset D_{r_2}$. Let $\mathcal{B}(D_r)$ denote the $\sigma$-algebra generated by $D_r,$ and $\mathcal{P}$ be the family of all probability measures on $(D_r,\mathcal{B}(D_r))$.  We also define probability measure set $M_\kappa$ as
\begin{equation*}
  M_\kappa:=\{\mu\in M_0;\mu^{(\kappa)}:=\int_{-\infty}^{0}e^{-\kappa\theta}\mu(d\theta)<\infty\}
\end{equation*}
for any $\kappa>0,$ where $M_0$ is the set of probability measures on $(-\infty,0].$
\par
Let $\nu$ denote a  given  $\si$-finite intensity measure of a L\'evy process  on $\mathbb{R}^n$. Recall that $\nu(\{0\})=0$ and $\int_{\mathbb{R}^n}(|z|^2\wedge1)\nu(dz)<\infty$.
Let $N: \cB(\RR^n\times\RR^+) \times \Omega\rightarrow \bar{\NN}$ be the time homogeneous Poisson random measure with intensity measure $\nu$. Again $\tilde{N} (dz,dt) = N(dz,dt) - \nu(dz)dt$ denotes the  compensated Poisson random measure associated to $N$.
\vskip 0.2cm
Let us point out that (as shown by e.g., \cite[Theorems 4.23 and 6.8]{PZ 2007}) in this case  $$L(t)=\int_0^t\int_{0<|z|\leq1}z\tilde{N}(dz,ds) + \int_0^t\int_{|z|>1} zN(dz,ds), t\geq0$$ defines an $\RR^n$-valued L\'evy process.

Consider an $n$-dimensional FSDE
\begin{equation}\label{yang-0}
dx(t)=f(x_t)dt+dL(t)
\end{equation}
with the initial value $\xi\in D_r$, where $x_t=x_t(\theta)=:\{x(t+\theta),-\infty<\theta\leq 0\},$ which sometimes we also denote by $x_t^{\xi}$ to indicate the initial value is $\xi$.  $f:D_r\rightarrow \mathbb{R}^n$ is Borel measurable. We also need to consider the following equations:
\begin{eqnarray}\label{yang-5}
  &&dX^{\epsilon}(t)= f(X^{\epsilon}_t)dt+\int_{0<\|z\|\leq \epsilon}z \tilde{N}(dt,dz),~~\epsilon\in(0,1)
\end{eqnarray}
and
\begin{eqnarray}\label{yang-4}
  &&dX(t)= f(X_t)dt.
\end{eqnarray}
In the following, we use the notation $x^{\xi}$, $X^{\epsilon,\xi}$ and $X^\xi$ to indicate the solutions of (\ref{yang-0}), (\ref{yang-5}) and (\ref{yang-4}) starting from $\xi$, respectively.
\vskip 0.2cm
We introduce the following assumption.

\begin{ass}\label{A}
\begin{itemize}
  \item[{\bf (A0)}]  $\nu$ is symmetric, i.e., $\nu(C)=\nu(-C)$, $\forall C\in\mathcal{B}(\mathbb{R}^n)$.

 \item[{\bf (A1)}] For any $\xi \in D_r$, there exist unique global solutions $x^\xi$, $X^{\epsilon,\xi}$ and $X^\xi$ to (\ref{yang-0}), (\ref{yang-5}) and (\ref{yang-4}), respectively, satisfying
     \begin{itemize}
    \item[(A1-1)] For any $\xi \in D_r$, $\lim_{t \rightarrow \infty} \| X^\xi_t \|_r = 0$.

 \item[(A1-2)] For any $t>0$ and $\xi \in D_r$, $\lim_{\epsilon \rightarrow 0} \|X^{\epsilon,\xi}_t - X^{\xi}_t \|_r = 0$ in probability.

\end{itemize}

\end{itemize}
\end{ass}
The main result of this paper is as follows:
 \begin{thm}\label{thm1}
 Assume that Assumption \ref{A} holds, then $\{x^\xi\}_{\xi\in D_r}$ is weakly irreducible to zero.
 \end{thm}
 \begin{proof}
 Fix $\xi\in D_r$ and arbitrary $\kappa>0$, according to Assumption (A1-1), there exists $\tilde T>0$ such that for any $T>\tilde T$
 
 \begin{equation}\label{yang-yang-1}
 \| X^\xi_T \|_r \leq \frac{\kappa}{2}.
 \end{equation}
Assumption (A1-2) implies that for any $T>\tilde T$, there exists $\tilde{\epsilon}=\tilde{\epsilon}(T)>0$ such that

\begin{equation}\label{yang-yang-2}
P(\|X^{\tilde{\epsilon},\xi}_T - X^{\xi}_T \|_r \leq\frac{\kappa}{4})\geq \frac{1}{2}.
\end{equation}

Let $\tau_{\tilde{\epsilon}}^1=\inf\{t\geq 0:\int_0^t\int_{|z|>\tilde{\epsilon}}N(dz,ds) \geq 1\}$. Then $\tau_{\tilde{\epsilon}}^1$ has the exponential distribution with parameter $\nu(|z|>\tilde{\epsilon})<\infty$. In particular,
\begin{eqnarray}\label{eq Zhai 03}
\mathbb{P}(\tau_{\tilde{\epsilon}}^1> s)=e^{-\nu(|z|>\tilde{\epsilon})s}  \text{ and}\,\, \PP(0<\tau_{\tilde{\epsilon}}^1< \infty) =1.
\end{eqnarray}
By virtue of Assumption (A0), we see that
$$L(t)=\int_0^t\int_{0<|z|\leq \tilde{\epsilon}}z\tilde{N}(dz,ds) + \int_0^t\int_{|z|_H>\tilde{\epsilon}} zN(dz,ds),\,\, t\geq0 .$$
Therefore, for any $\xi\in D_r$, $\{x^\xi_t,t\in[0,\tau_{\tilde{\epsilon}}^1)\}$ is the unique solution to (\ref{yang-5}) with the initial data $\xi$ on $t\in[0,\tau_{\tilde{\epsilon}}^1)$. We also note that $\sigma\{X^{\tilde{\epsilon},\xi}_t,t\geq 0\}$ and $\sigma\{\tau_{\tilde{\epsilon}}^1\}$ are independent.

By (\ref{yang-yang-1}), (\ref{yang-yang-2}) and (\ref{eq Zhai 03}), for any fixed $\xi\in D_r$, $\kappa>0$ and $\forall ~T>\tilde T$
\begin{eqnarray*}
&&P(\|x^{\xi}_T\|_r\leq \kappa)\nonumber\\
&=&P(\|x^{\xi}_T-X^{\tilde{\epsilon},\xi}_T +X^{\tilde{\epsilon},\xi}_T- X^{\xi}_T +X^{\xi}_T\|_r\leq \kappa)\nonumber\\
&\geq& P(\|x^{\xi}_T-X^{\tilde{\epsilon},\xi}_T +X^{\tilde{\epsilon},\xi}_T- X^{\xi}_T +X^{\xi}_T\|_r\leq \kappa,\tau_{\tilde{\epsilon}}>T)\nonumber\\
&=&P(\|X^{\tilde{\epsilon},\xi}_T- X^{\xi}_T +X^{\xi}_T\|_r\leq \kappa,\tau_{\tilde{\epsilon}}>T)\nonumber\\
&=&P(\|X^{\tilde{\epsilon},\xi}_T- X^{\xi}_T +X^{\xi}_T\|_r\leq \kappa)P(\tau_{\tilde{\epsilon}}>T)\nonumber\\
&\geq&P(\|X^{\tilde{\epsilon},\xi}_T - X^{\xi}_T \|_r \leq\frac{\kappa}{4})P(\tau_{\tilde{\epsilon}}>T)\nonumber\\
&> &0.
\end{eqnarray*}

The proof of Theorem \ref{thm1} is complete.
 \end{proof}
\section{Applications}
In this section, as an application of the main result  of Theorem \ref{thm1},  we obtain the weak  irreduciblility of a class of nonlinear SDDEs with weakly dissipative coefficients driven by additive pure jump noise.

Let us formulate the  assumptions on the coefficients.

(H1) Assume that the coefficients $f$ satisfy the local Lipschitz condition, that is, for any $k>0,$ there exists a $c_k$ such that
\begin{equation}
|f(\varphi)-f(\psi)|^2\leq c_k||\varphi-\psi||_r^2
\end{equation}
 for those $\varphi,~\phi$ $\in$ $D_r$ with $||\varphi||_r\vee||\phi||_r\leq k$.

(H2) There exist nonnegative constants $\bar\lambda_1,$ $\bar\lambda_2,$ $\bar K_1,$ $\bar K_2$, and probability measures $\mu_1\in M_{(2\vee q_1)r},$ $\mu_2\in M_{2r}$, such that for any $\varphi$, $\psi\in D_r$,
\begin{eqnarray}\label{410}
&&2\langle\varphi(0)-\psi(0),f(\varphi)-f(\psi)\rangle\nonumber\\
&\leq&\bar K_1\int_{-\infty}^{0}H(\varphi(\theta),\psi(\theta))d\mu_1(\theta)-\bar K_2H(\varphi(0),\psi(0))\nonumber\\
&&\quad +\bar\lambda_1\int_{-\infty}^{0}|\varphi(\theta)-\psi(\theta)|^2d\mu_2(\theta)-\bar\lambda_2|\varphi(0)-\psi(0)|^2,
\end{eqnarray}
where $H$ is a continuous nonnegative function defined on $\mathbb{R}^n \times \mathbb{R}^n$ such that $H(x,x)=0$ and $H(x,y)\leq K(|x|^{q_1}+|y|^{q_2})$, $q_1\wedge q_2>0$.


In the sequel, the symbol $C$ will denote a positive generic constant whose value may
change from line to line.


We consider the weak  irreduciblility of (\ref{yang-0}). The main result is stated as follows.

\begin{prop}\label{prop 1}
Assume that (H1), (H2) hold and $\nu$ is symmetric, if $\bar\lambda_2-2r-\bar\lambda_1\mu_2^{(2r)}\geq0$ and $\bar K_2-\bar K_1\mu_1^{(2r)}\geq0$, then the solution $\{x^{\xi}\}_{\xi\in D_r}$ of (\ref{yang-0}) is weakly irreducible to zero.
\end{prop}

%
%
%
%
%

\begin{proof}
We will apply   Theorem \ref{thm1}, that is, we  verify Assumptions (A0)-(A1).  Since (A0) is imposed,  we only need to verify Assumption (A1).

Since $f$ is local Lipschitz continuous and satisfies (\ref{410}), there exist unique global solutions to (\ref{yang-5}) and (\ref{yang-4}), respectively. For more details, one can invoke Theorem V.7 in Protter \cite{Protter} and Remark 1.1 in \cite{deng1}. The unique global solution to (\ref{yang-0}) can be obtained by the solution to (\ref{yang-5}) with $\epsilon=1$ and the so-called interlacing technique; see \cite{Liuwei}.

Note that for any $t\geq 0$,
\begin{eqnarray*}
||X_t^{\xi}||_r\leq e^{-rt}\sup_{0\leq s\leq t}|X(s)|+e^{-rt}||\xi||_r.
\end{eqnarray*}
So to prove  (A1-1), we only need to prove $\sup_{0\leq t<\infty}|X(t)|\leq C$.

Choose $\psi=0$ in (H2), by the Young inequality, for any $0<\delta<2r$, we derive
\begin{eqnarray}\label{yang--1}
&&2\langle\varphi(0),f(\varphi)\rangle\nonumber\\
&=&2\langle\varphi(0),f(\varphi)-f(0)\rangle+2\varphi^T(0)f(0)\nonumber\\
&\leq&2\langle\varphi(0),f(\varphi)-f(0)\rangle+\frac{1}{\delta}|f(0)|^2+\delta|\varphi(0)|^2\nonumber\\
&\leq&\frac{1}{\delta}|f(0)|^2+\bar K_1\int_{-\infty}^{0}H(\varphi(\theta),0)d\mu_1(\theta)-\bar K_2H(\varphi(0),0)\nonumber\\
&&+\bar\lambda_1\int_{-\infty}^{0}|\varphi(\theta)|^2d\mu_2(\theta)-\bar\lambda_2|\varphi(0)|^2+\delta|\varphi(0)|^2.
\end{eqnarray}
By the chain rule and (\ref{yang--1}),  for every $t>0$,
  \begin{eqnarray}\label{89}
&&e^{2rt}|X(t)|^2\nonumber\\
&=&|\xi(0)|^2+\int_{0}^{t}(2re^{2rs}|X(s)|^2+2e^{2rs}\langle X(s),f(X_s)\rangle)ds\nonumber\\
  &\leq& |\xi(0)|^2+\int_{0}^{t}e^{2rs}\Big(2r|X(s)|^2+L+\bar K_1 \int_{-\infty}^0\widetilde H(X(s+\theta))d\mu_1(\theta)-\bar K_2\widetilde H(X(s))\nonumber\\
  &&+\bar\lambda_1\int_{-\infty}^{0}|X(s+\theta)|^2d\mu_2(\theta)-\bar\lambda_2|X(s)|^2+\delta|X(s)|^2\Big)ds,
\end{eqnarray}
where $L=\frac{1}{\delta}|f(0)|^2$, $\widetilde H(x)=H(x,0)$. Note that
\begin{eqnarray}\label{00002}
  &&\int_{0}^{t}\int_{-\infty}^{0}e^{2rs}|X(s+\theta)|^2d\mu_2(\theta)ds\nonumber\\
  &&=\int_{0}^{t}\Big[\int_{-\infty}^{-s}e^{2r s}|X(s+\theta)|^2d\mu_2(\theta)+\int_{-s}^{0}e^{2rs}|X(s+\theta)|^2d\mu_2(\theta)\Big]ds\nonumber\\
  &&=\int_{0}^{t}ds\int_{-\infty}^{-s}e^{2r(s+\theta)}|X(s+\theta)|^2e^{-2r(s+\theta)+2rs}d\mu_2(\theta)+\int_{-t}^{0}d\mu_2(\theta)\int_{-\theta}^{t}e^{2rs}|X(s+\theta)|^2ds\nonumber\\
   &&\leq t||\xi||_r^2\int_{-\infty}^{-s}e^{-2r\theta}d\mu_2(\theta)+\int_{-t}^{0}e^{-2r\theta}d\mu_2(\theta)\int_{-\theta}^{t}e^{2r(s+\theta)}|X(s+\theta)|^2ds\nonumber\\
    &&\leq t||\xi||_r^2\int_{-\infty}^{0}e^{-2r\theta}\mu_2(\theta)+\int_{-\infty}^{0}e^{-2r\theta}d\mu_2(\theta)\int_{0}^{t}e^{2rs}|X(s)|^2ds\nonumber\\
 &&= t||\xi||_r^2\mu_2^{(2r)}+\mu_2^{(2r)}\int_{0}^{t}e^{2rs}|X(s)|^2ds
\end{eqnarray}
and
\begin{eqnarray}\label{881}
  &&\int_{0}^{t}\int_{-\infty}^{0}e^{2rs}\widetilde H(X(s+\theta))d\mu_1(\theta)ds\nonumber\\
  &&=\int_{0}^{t}\Big[\int_{-\infty}^{-s}e^{2rs}\widetilde H(X(s+\theta))d\mu_1(\theta)+\int_{-s}^{0}e^{2rs}\widetilde H(X(s+\theta))d\mu_1(\theta)\Big]ds\nonumber\\
  &&=\int_{0}^{t}e^{2rs}ds\int_{-\infty}^{-s}\widetilde H(e^{r(s+\theta)}X(s+\theta)e^{-r(s+\theta)})d\mu_1(\theta)+\int_{-t}^{0}d\mu_1(\theta)\int_{-\theta}^{t}e^{2rs}\widetilde H(X(s+\theta))ds\nonumber\\
   &&\leq\int_{0}^{t}\int_{-\infty}^{-s}K||\xi||_r^{q_1}e^{-q_1r(s+\theta)}e^{2rs}d\mu_1(\theta)ds+\int_{-t}^{0}e^{-2r\theta}d\mu_1(\theta)\int_{-\theta}^{t}e^{2r(s+\theta)}\widetilde H(X(s+\theta))ds\nonumber\\
    &&\leq \frac{K}{-(q_1r-2r)}||\xi||_r^{q_1}\mu_1^{(q_1r)}(e^{-(q_1r-2r)t}-1)+\mu_1^{(2r)}\int_{0}^{t}e^{2rs}\widetilde H(X(s))ds.
\end{eqnarray}
According to our assumption,
\begin{equation*}
\bar \lambda_2-2r-\bar\lambda_1\mu_2^{(2r)}\geq0,\quad \bar K_2-\bar K_1\mu_1^{(2r)}\geq0,
\end{equation*}
this together with (\ref{89}), (\ref{00002}), (\ref{881}) yields that
\begin{eqnarray}\label{91}
&&e^{2rt}|X(t)|^2\nonumber\\
&\leq&|\xi(0)|^2+\frac{L}{2r}(e^{2rt}-1)+\bar\lambda_1t||\xi||_r^2\mu_2^{(2r)}\nonumber\\
&&+\frac{\bar K_1K}{-(q_1r-2r)}||\xi||_r^{q_1}\mu_1^{(q_1r)}(e^{-(q_1r-2r)t}-1)+\delta\int_{0}^{t}e^{2rs}|X(s)|^2ds\nonumber\\
&\leq&C(1+e^{2rt})+\delta\int_{0}^{t}e^{2rs}|X(s)|^2ds.
\end{eqnarray}
Applying the Gronwall inequality, we have

\begin{eqnarray*}
 e^{2rt}|X(t)|^2\leq Ce^{2rt},
\end{eqnarray*}
which gives that $$\sup_{0\leq t< \infty}|X(t)|\leq C.$$

\noindent Next, we verify (A1-2).

Note that
\begin{eqnarray*}
  &&X^{\epsilon,\xi}(t)-X^{\xi}(t)=\int_0^t f(X^{\epsilon,\xi}_s)-f(X^{\xi}_s)ds+\int_0^t\int_{0<|z|\leq \epsilon}z \tilde{N}(ds,dz).
\end{eqnarray*}

Applying the Ito formula and (H2) gives
\begin{eqnarray}\label{yang-ineq3}
 &&e^{2rt}|X^{\epsilon,\xi}(t)-X^{\xi}(t)|^2\nonumber\\
 &=&\int_{0}^{t}(2re^{2rs}|X^{\epsilon,\xi}(s)-X^{\xi}(s)|^2+2\langle X^{\epsilon,\xi}(s)-X^{\xi}(s),f(X^{\epsilon,\xi}_s)-f(X^{\xi}_s)\rangle)ds\nonumber\\
 &&+2\int_0^t\int_{0<|z|\leq\epsilon}e^{2rs}\langle X^{\epsilon,\xi}(s)-X^{\xi}(s),z\rangle\tilde{N}(ds,dz)+\int_0^t\int_{0<|z|\leq\epsilon}e^{2rs}|z|^2{N}(ds,dz)\nonumber\\
 &\leq&\int_{0}^{t}e^{2rs}\Big(2r|X^{\epsilon,\xi}(s)-X^{\xi}(s)|^2+\bar\lambda_1\int_{-\infty}^{0}|X^{\epsilon,\xi}(s+\theta)-X^{\xi}(s+\theta)|^2d\mu_2(\theta)\nonumber\\
  &&-\bar\lambda_2|X^{\epsilon,\xi}(s)-X^{\xi}(s)|^2+\bar K_1 \int_{-\infty}^0H(X^{\epsilon,\xi}(s+\theta),X^{\xi}(s+\theta))d\mu_1(\theta)\nonumber\\
  &&-\bar K_2H(X^{\epsilon,\xi}(s),X^{\xi}(s))\Big)ds+2\int_0^t\int_{0<|z|\leq\epsilon}e^{2rs}\langle X^{\epsilon,\xi}(s)-X^{\xi}(s),z\rangle\tilde{N}(ds,dz)\nonumber\\
  &&+\int_0^t\int_{0<|z|\leq\epsilon}e^{2rs}|z|^2{N}(ds,dz).
\end{eqnarray}
According to (\ref{00002}), (\ref{881}) and $H(\xi(\theta),\xi(\theta))=0$,
\begin{equation}\label{yang-ineq4}
\int_{0}^{t}\int_{-\infty}^{0}e^{2rs}|X^{\epsilon,\xi}(s+\theta)-X^{\xi}(s+\theta)|^2d\mu_2(\theta)ds\leq \mu_2^{(2r)}\int_{0}^{t}e^{2rs}|X^{\epsilon,\xi}(s)-X^{\xi}(s)|^2ds
\end{equation}
and
\begin{equation}\label{yang-ineq5}
\int_{0}^{t}\int_{-\infty}^{0}e^{2rs} H(X^{\epsilon,\xi}(s+\theta),X^{\xi}(s+\theta))d\mu_1(\theta)ds\leq \mu_1^{(2r)}\int_{0}^{t}e^{2rs}H(X^{\epsilon,\xi}(s),X^{\xi}(s))ds.
\end{equation}
By (\ref{yang-ineq3}), (\ref{yang-ineq4}), (\ref{yang-ineq5}) and $\bar\lambda_2-2r-\bar\lambda_1\mu_2^{(2r)}\geq0$ and $\bar K_2-\bar K_1\mu_1^{(2r)}\geq0$, we have
\begin{eqnarray}
 &&e^{2rt}|X^{\epsilon,\xi}(t)-X^{\xi}(t)|^2\nonumber\\
&\leq&-(\bar\lambda_2-2r-\bar\lambda_1\mu_2^{(2r)})\int_{0}^{t}e^{2r s}|X^{\epsilon,\xi}(s)-X^{\xi}(s)|^2ds\nonumber\\
   &&-(\bar K_2-\bar K_1\mu_1^{(2r)})\int_{0}^te^{2rs}H(X^{\epsilon,\xi}(s),X^{\xi}(s))ds\nonumber\\
   &&+2\int_0^t\int_{0<|z|\leq\epsilon}e^{2rs}\langle X^{\epsilon,\xi}(s)-X^{\xi}(s),z\rangle\tilde{N}(ds,dz)\nonumber\\
  &&+\int_0^t\int_{0<|z|\leq\epsilon}e^{2rs}|z|^2{N}(ds,dz)\nonumber\\
&\leq& 2\int_0^t\int_{0<|z|\leq\epsilon}e^{2rs}\langle X^{\epsilon,\xi}(s)-X^{\xi}(s),z\rangle\tilde{N}(ds,dz)\nonumber\\
  &&+\int_0^t\int_{0<|z|\leq\epsilon}e^{2rs}|z|^2{N}(ds,dz).
\end{eqnarray}
Then by the stochastic Gronwall inequality (see Lemma 3.7 in \cite{XZ}), for any $0<q<p<1$
\begin{eqnarray*}
\mathbb{E}\Big(\sup_{0\leq s\leq t}e^{2rs}|X^{\epsilon,\xi}(t)-X^{\xi}(t)|^2\Big)^q\leq\frac{p}{p-q}\Big(\int_0^t\int_{0<|z|\leq \epsilon}e^{2rs}|z|^2\nu(dz)ds\Big)^q.
\end{eqnarray*}
Since

\begin{equation*}
\|X^{\epsilon,\xi}_t-X^{\xi}_t\|^2_r\leq e^{-2rt}\sup_{0\leq s\leq t}e^{2rs}|X^{\epsilon,\xi}(t)-X^{\xi}(t)|^2,
\end{equation*}
we have

\begin{eqnarray}
\mathbb{E}\|X^{\epsilon,\xi}_t-X^{\xi}_t\|_r^{2q}\leq e^{-2qrt}\frac{p}{p-q}\Big(\int_0^t\int_{0<|z|\leq \epsilon}e^{2rs}|z|^2\nu(dz)ds\Big)^q.
\end{eqnarray}
 Observing that, in this setting, the convergence
is stronger than the one (in probability) required in condition (A1-2),
so (A1-2) holds. So the proof of Proposition \ref{prop 1} is also complete.
\end{proof}
\vspace*{4pt}\noindent {\bf Example.}
Let us consider a functional SDE with nonlinear distribution delay,
  \begin{equation}\label{non}
    dx(t)=[1-2x(t)-2x^3(t)+\int_{-\infty}^{0}x^2(t+\theta)d\mu_1(\theta)]dt+dL(t)
  \end{equation}
  on $t\geq 0$ with initial data $\{\xi(t):-\infty<t\leq 0\}\in D_r$, where $L(t)$ is a one-dimensional jump process. To verify (\ref{410}), note that 
   \begin{eqnarray*}
  &&2\langle\varphi(0)-\psi(0),f(\varphi)-f(\psi)\rangle\nonumber\\
  &=&2(\varphi(0)-\psi(0))\Big(-2(\varphi(0)-\psi(0))-2(\varphi^3(0)-\psi^3(0))+\int_{-\infty}^{0}\varphi^2(\theta)-\psi^2(\theta)d\mu_1(\theta)\Big)\nonumber\\
  &\leq&-3(\varphi(0)-\psi(0))^2-4(\varphi(0)-\psi(0))^2\big(\varphi^2(0)+\varphi(0)\psi(0)+\psi^2(0)\big)\nonumber\\
  &&+\int_{-\infty}^{0}(\varphi(\theta)-\psi(\theta))^2(\varphi(\theta)+\psi(\theta))^2d\mu_1(\theta)\nonumber\\
  &\leq&-3(\varphi(0)-\psi(0))^2-2(\varphi(0)-\psi(0))^2\big(2\varphi^2(0)+2\varphi(0)\psi(0)+2\psi^2(0)\big)\nonumber\\
  &&+\int_{-\infty}^{0}(\varphi(\theta)-\psi(\theta))^2\big(2\varphi^2(\theta)+2\varphi(\theta)\psi(\theta)+2\psi^2(\theta)\big)d\mu_1(\theta)\nonumber.
\end{eqnarray*}
Choose $\bar\lambda_1=0,$ $\bar\lambda_2=3$, $\bar K_1=1,$ $\bar K_2=2$ and $H(x,y)=(x-y)^2(2x^2+2xy+2y^2)$, then $H(x,y)\leq 12(x^4+y^4).$
 If $3-2r\geq 0$, $2-\mu_1^{(2r)}\geq0$ and $\mu_1\in M_{4r}$, according to Proposition \ref{prop 1},  the solution $\{x^{\xi}\}_{\xi\in D_r}$ of (\ref{non}) is weakly irreducible to zero.

\vskip 0.4cm
\noindent{\bf Acknowledgements.}
This work is partially supported by National Natural Science Foundation
 of China (Nos. 12401175, 12371151), Scientific Research Fund of Zhejiang Provincial
Education Department (No. Y202455933), the Natural Science Foundation of Zhejiang Provincial (No. LQN25A010020), the Fundamental Research Funds for the Central
Universities (No. JZ2025HGTB0172).
\vskip 0.4cm

\noindent{\bf Statements and Declarations} This research has no experiments. The authors declare that they have no
conflict of interest.

\end{document}